\documentclass[10pt,%
               a4paper,%
               oneside,%
							 reqno,
							 english
               ]{smfart}

\usepackage{amssymb}
\usepackage{tensor,slashed}
\usepackage{smfthm}
\usepackage[all,cmtip]{xy}

\usepackage[T1]{fontenc}

\usepackage{lmodern}
\usepackage[utf8]{inputenc}

\usepackage[english]{babel}

\usepackage[dvipsnames]{xcolor}

\usepackage{tikz}
\usetikzlibrary{decorations.pathmorphing}

\usepackage{microtype}
\usepackage{hyperref}

\DeclareRobustCommand{\SkipTocEntry}[5]{}

\newcommand{\numberset}{\mathbb}
\newcommand{\N}{\numberset{N}} 
\newcommand{\Z}{\numberset{Z}}

\newcommand{\A}{\mathcal{A}}
\newcommand{\F}{\mathcal{F}}
\let\P\relax
\let\M\relax
\newcommand{\P}{\mathcal{P}}
\newcommand{\M}{\mathcal{M}}
\newcommand{\Cc}{\mathcal{C}}

\DeclareMathOperator{\Hom}{Hom}

\DeclareMathOperator{\Tot}{Tot^{\oplus}}

\author{Valerio Proietti}
\address{Department of Mathematical Sciences, Københavns Universitet, Universitetsparken 5, 2100 København Ø, Denmark}
\email{valerio@math.ku.dk}

\title[A note on homology for Smale spaces]{A note on homology for Smale spaces}

\begin{document}
\frontmatter

\begin{abstract}
We collect three observations on the homology for Smale spaces defined by Putnam. The definition of such homology groups involves four complexes. It is shown here that a simple convergence theorem for spectral sequences can be used to prove that all complexes yield the same homology. Furthermore, we introduce a simplicial framework by which the various complexes can be understood as suitable ``symmetric'' Moore complexes associated to the simplicial structure. The last section discusses projective resolutions in the context of dynamical systems. It is shown that the projective cover of a Smale space is realized by the system of shift spaces and factor maps onto it.
\end{abstract}

\thanks{Supported by the Danish National Research Foundation through the Centre for Symmetry and Deformation (DNRF92).}

\maketitle
\setcounter{tocdepth}{2}
\vspace*{-1ex}		
\tableofcontents

\mainmatter

\addtocontents{toc}{\SkipTocEntry}\section*{Introduction and main results}

When Steven Smale initiated his study of smooth maps on manifolds, he defined the notion of Axiom A diffeomorphism \cite{smale:A}. The main condition is
that the map, restricted to its set of non-wandering points, has a hyperbolic
structure. The non-wandering set of these systems can be
canonically decomposed into finitely many disjoints sets, called \emph{basic sets}, each of which is irreducible in a certain sense. One of Smale's great insights was that, even though one began with a smooth system, the non-wandering set itself would not usually be a
submanifold, but rather an object of fractal-like nature. This can be taken as a motivation for moving from the smooth category to the topological one.

Smale spaces were introduced by Ruelle as a purely topological description of the basic sets of Smale's Axiom A diffeomorphisms \cite{ruelle:thermo}. In this paper we consider the homology theory for Smale spaces introduced by Putnam in \cite{put:HoSmale}. This can be viewed as a solution to Smale's problem of classifying Axiom A systems by relatively simple combinatorial data, in the same fashion that Morse-Smale systems could be described.


Shifts of finite type are the zero dimensional examples of Smale spaces and are the basic building blocks of the theory. Putnam's homology can be viewed as a far-reaching generalization of Krieger's dimension groups for shifts of finite type \cite{krieger:inv}. In the preliminaries of this paper, we review the notion of Krieger's invariant and explain its connection to $K$-theory by examining the stable and unstable equivalence relations which define the associated $C^*$-algebras (this is a well-known result, here it is simply expressed in a slightly unusual form, see Theorem \ref{thm:dimCgroup}).

There are many interesting and open questions concerning Putnam's homology for Smale spaces. In the literature, computations of the homology groups have been done mostly by resorting to the definition, e.g., \cite[Chapter 7]{put:HoSmale}. It is desirable to have some machinery, as it occurs with algebraic topology, which would aid in these calculations by appealing to techniques such as long exact sequences, excision, etc. 

Exact analogues are at the moment not so clear, but it is reasonable that an alternative, perhaps more conceptual definition of the homology could shed some light on these issues. Moreover, this could also lead to clarifying the relations with Čech cohomology and $K$-theory (beyond the case of shifts of finite type). More on these questions can be found in \cite[Chapter 8]{put:HoSmale}. This paper started as an effort to research in this direction.

The technical definition of Putnam's homology groups involves four bicomplexes \cite[Chapter 5]{put:HoSmale}. Only three of these are shown to be quasi-isomorphic, leaving out the largest (but perhaps most natural) double complex, which has a clear connection to $K$-theory. The first result of this paper fills this gap by showing, thanks to a simple convergence theorem for spectral sequences, that this double complex also yields the same homology groups. The formal statement is given in Corollary \ref{cor:conv}.

Section \ref{sec:sim} is concerned with proving a collection of results that are already proved in Putnam's memoir, by taking a slightly different and somewhat more unified perspective. The key observation stems from the simplicial nature of the homology theory for Smale spaces: a given Smale space is suitably ``replaced'' by a bisimplicial shift of finite type, to which Krieger's invariant is applied to get (in conjunction with the Dold-Kan correspondence) a bicomplex which defines the homology groups of interest. 

From this viewpoint, the different variants of this bicomplex appear as the associated Moore complexes (i.e., the normalized chain complexes). There is also an action of the symmetric group which is exploited to obtain all of Putnam's complexes as ``reduced'' complexes with respect to this mixed simplicial-symmetric structure. The main result in this section, proved as application of these methods, is Theorem \ref{thm:degthm}.

The last section introduces the concept of projectivity for dynamical systems and attempts to justify the definition of the homology theory for Smale spaces by drawing a parallel with sheaf cohomology. The main result here is that the projective cover of a Smale space can be defined as a certain projective limit over the symbolic presentations for the given space. The rigorous statement is found in Theorem \ref{thm:proj}.

Most of the conventions and notations in this paper are taken directly from \cite{put:HoSmale}. No attempt is made to put the results in broader context or expand on detail. For these reasons the reader is advised to have a copy of Putnam’s \emph{A Homology Theory for Smale Spaces} \cite{put:HoSmale} handy.

\addtocontents{toc}{\SkipTocEntry}\section*{Acknowledgements}

My gratitude goes to Ian F. Putnam for many stimulating conversations and for the warm hospitality I received during my stay at the Department of Mathematics and Statistics of the University of Victoria. I would also like to thank Ryszard Nest for proposing to look into projective resolutions.

\section{Preliminaries}
\label{sec:prelims}

A Smale space $(X, \phi)$ is a dynamical system consisting of a homeomorphism $\phi$ on a compact metric space $(X,d)$ such that the space is locally the product of a coordinate that
contracts under the action of $\phi$ and a coordinate that expands under the action of $\phi$. The precise definition requires the definition of a bracket map satisfying certain axioms \cite{put:HoSmale,ruelle:thermo}. 

The most essential feature of Smale spaces is given by the definition of two equivalence relations, named respectively \emph{stable} and \emph{unstable}, which reads as follows:
\begin{itemize}
\item given $x,y\in X$, we say they are \emph{stably equivalent} if
\[
\lim_{n\to +\infty} d(\phi^{n}(x),\phi^{n}(y)) = 0;
\]
\item given $x,y\in X$, we say they are \emph{unstably equivalent} if
\[
\lim_{n\to +\infty} d(\phi^{-n}(x),\phi^{-n}(y)) = 0.
\]
\end{itemize}

The orbit of $x\in X$ under the stable (respectively unstable) equivalence relation is called the \emph{global stable} (resp. \emph{unstable}) set and is denoted $X^s(x)$ (resp. $X^u(x)$). Given a small enough $\epsilon >0$, \emph{local} stable and unstable sets are also defined, and they are denoted respectively $X^s(x,\epsilon)$ and $X^u(x,\epsilon)$. They provide the local product structure in the following sense: each $x\in X$ admits an open neighborhood which is homeomorphic (via the bracket map) to the product $X^u(x,\epsilon)\times X^s(x,\epsilon)$. Local and global sets are related through the following identities:
\begin{align*}
X^s(x)=&\bigcup_{n\geq 0} \phi^{-n}(X^s(\phi^n(x),\epsilon))\\
X^u(x)=&\bigcup_{n\geq 0} \phi^{n}(X^s(\phi^{-n}(x),\epsilon)).
\end{align*}

Let $(X,\phi)$ be a Smale space. We will assume that $(X,\phi)$ is \emph{non-wandering}, so that there exists an \emph{$s/u$-bijective pair} $\pi=(Y,\psi,\pi_s,Z,\zeta,\pi_u)$ (see \cite[Section 2.6]{put:HoSmale} for this notion). Recall from \cite[Sections 2.5 and 2.6]{put:HoSmale} that we can assume $Y$ and $Z$ to be non-wandering, and also $\pi_s$ and $\pi_u$ to be finite-to-one. 

We define a subshift of finite type for each $L,M\geq 0$,
\begin{align*}
\Sigma_{L,M}(\pi)=\{(y_0,\dots,y_L,z_0,\dots,z_M)\mid& y_l\in Y, z_m\in Z,\\
 &\pi_s(y_l)=\pi_u(z_m), 0\leq l\leq L, 0\leq m\leq M\}.
\end{align*}
We have maps
\begin{align}\label{eq:deltalm}
\delta_l\colon  \Sigma_{L,M}&\to \Sigma_{L-1,M}\\\notag
\delta_{,m}\colon  \Sigma_{L,M+1}&\to \Sigma_{L,M}
\end{align}
which delete respectively entries $y_l$ and $z_m$. Theorem 2.6.13 in \cite{put:HoSmale} asserts that the maps $\delta_l$ are $s$-bijective and the maps $\delta_{,m}$ are $u$-bijective (these will be defined shortly). 

Given a subshift of finite type $\Sigma$, we can associate to it an abelian group, denoted $D^s(\Sigma)$, defined in \cite[Chapter 3]{put:HoSmale} (see also \cite{krieger:inv}). It will be called the (stable) \emph{dimension group} of $\Sigma$. This construction is covariant for $s$-bijective maps and contravariant for $u$-bijective maps \cite[Sections 3.4 and 3.5]{put:HoSmale}. We summarize here these definitions:

\begin{defi}
Let $f\colon (X,\phi)\to (Y,\psi)$ be a map of Smale spaces. Consider for each $x\in X$ the restrictions
\begin{align}\label{eq:sbij}
f\colon X^s(x)&\to Y^s(f(x))\\\label{eq:ubij}
f\colon X^u(x)&\to Y^u(f(x)).
\end{align}
\begin{itemize}
\item If \eqref{eq:sbij} is injective, we say that $f$ is \emph{$s$-resolving}. If it is injective and surjective, then we say $f$ is \emph{$s$-bijective}.
\item If \eqref{eq:ubij} is injective, we say that $f$ is \emph{$u$-resolving}. If it is injective and surjective, then we say $f$ is \emph{$u$-bijective}.
\end{itemize}
\end{defi}

\subsection{Dimension groups}

Let us start with the definition of Krieger's dimension groups.

\begin{defi}\label{def:dimgrp}
Let $(\Sigma, \sigma)$ be a subshift of finite type. For $e\in \Sigma$, consider the family of compact open subsets in the stable orbit $\Sigma^s(e)$ and denote it by $CO^s(\Sigma,\sigma,e)$. Define $CO^s(\Sigma,\sigma)=\cup_{e\in\Sigma} CO^s(\Sigma,\sigma,e)$. Let $\sim$ be the smallest equivalence relation such that, for $E,F\in CO^s(\Sigma,\sigma)$, we have 
\begin{itemize}
\item $E\sim F$ if $[E,F]=E,[F,E]=F$, assuming both sets are defined;
\item $E\sim F$ if and only if $\sigma(E)\sim \sigma(F)$.
\end{itemize}
We define $D^s(\Sigma,\sigma)$ (abbreviated $D^s(\Sigma)$) to be the free abelian group on the $\sim$-equivalences $[E]$, modulo the subgroup generated by $[E\cup F]-[E]-[F]$, where $E,F$ belong to $CO^s(\Sigma,\sigma)$ and $E\cap F=\emptyset$.
\end{defi}

There is a definition of $D^u(\Sigma,\sigma)$, which is left to the imagination of the reader, since it won't be used in the rest of this paper.

It is easy to see that, in the construction above, it is sufficient to consider clopens lying in the \emph{local} stable sets.

\begin{lemm}
Define a family of sets $CO^s_\epsilon(\Sigma,\sigma)$, composed of clopens $E\subseteq \Sigma^s(e,\epsilon)$ for some $e\in \Sigma$ and $\epsilon< 1/4$. Consider the abelian group $D^s_\epsilon(\Sigma,\sigma)$, defined as in Definition \ref{def:dimgrp}, but replacing $CO^s(\Sigma,\sigma)$ with $CO^s_\epsilon(\Sigma,\sigma)$. Then we have $D^s(\Sigma,\sigma)\cong D^s_\epsilon(\Sigma,\sigma)$.
\end{lemm}
\begin{proof}
Given $E\in CO^s(\Sigma,\sigma), E\subseteq \Sigma^s(f)$, there is a well-defined function $E\to \N$, defined assigning to $e\in E$ the minimum number $N(e)$ such that $e_n=f_n$ whenever $n\geq N(e)$. In other words, $N(e)$ is the minimum natural number such that 
\[
e\in \sigma^{-N(e)}(\Sigma^s(\sigma^{N(e)}(f),\epsilon)).
\]
By definition $E\cap \sigma^{-n}(\Sigma^s(\sigma^{n}(f),\epsilon))$ is clopen for each $n\in \N$, which implies the assignment $e\mapsto N(e)$ is continuous. Since $E$ is compact, there is $N(E)\in \N$ such that
\[
E\subseteq \sigma^{-N(E)}(\Sigma^s(\sigma^{N(E)}(f),\epsilon)).
\]
Therefore $E$ can be partitioned in a finite number of disjoint clopens $E_i$ with $E_i \in CO^s_\epsilon(\Sigma,\sigma)$. We conclude $[E]\in D^s_\epsilon(\Sigma,\sigma)$. All is left to show is the equivalence relation defining $D^s(\Sigma,\sigma)$ is determined within the clopens in $CO^s_\epsilon(\Sigma,\sigma)$. Let $E\sim F$ be sets in $CO^s(\Sigma,\sigma)$ and take $N$ to be the maximum between $N(E)$ and $N(F)$. By definition $E\sim F$ if and only if $\sigma^N(E)\sim \sigma^N(F)$, and of course $\sigma^N(E)$ and $\sigma^N(F)$ belong to local stable sets. This completes the proof.
\end{proof}

A consequence of the previous lemma is that we can illustrate the definition of dimension group by a simple figure (Figure \ref{fig:dimg}). 

\begin{figure}
\centering
\begin{tikzpicture}[font=\scriptsize]
\draw[step=5mm,gray,very thin] (0,0) grid (2,2);
\draw[->] (0,0) -- (2.2,0) node[anchor=north west] {$\Sigma^u(x)$};
\draw[->] (0,0) -- (0,2.2) node[anchor=north east] {$\Sigma^s(x)$};
\fill[black] (0mm,0mm) circle[radius=1.1pt] node[anchor=north east] {$x$};
\draw[-,thick,color=blue] (0.5,0.5) -- (0.5,1.5) node[anchor=north east] {\textcolor{black}{$E$}};
\draw[-,thick,color=red] (1.5,0.5) -- (1.5,1.5) node[anchor=north west] {\textcolor{black}{$F$}};
\fill[fill opacity=0.15,fill=yellow]  (0.5,0.5) -- (0.5,1.5) -- (1.5,1.5) -- (1.5,0.5) -- cycle;
\draw (5mm,1pt) -- (5mm,-1pt) node[anchor=north] {$e$};
\draw (15mm,1pt) -- (15mm,-1pt) node[anchor=north] {$f$};
\draw (1,1) node[above,font=\normalsize,] {$\sim$};
\end{tikzpicture}
\caption[Dimension groups]{In this figure, $E$ and $F$ are compact opens in $CO^s_\epsilon(\Sigma,\sigma)$, with $E\subseteq \Sigma^s(e,\epsilon)$ and $F\subseteq \Sigma^s(f,\epsilon)$. The shaded area in yellow indicates that $[E,F]=E,[F,E]=F$ and therefore $E$ and $F$ are equivalent sets.}\label{fig:dimg}
\end{figure}
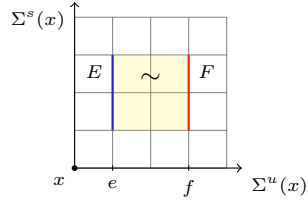

When $(\Sigma,\sigma)$ is non-wandering we can simplify the computation of the dimension group even further, because we can decompose $\Sigma$ in basic pieces as follows (see \cite[Theorem 2.1.13]{put:HoSmale}). 

\begin{theo}\label{thm:decomp}
Given a non-wandering Smale space $(X,\phi)$, there are closed pairwise disjoint sets $X_1,\dots,X_n$ and a permutation $\alpha\in S_n$ such that $\phi(X_i)=X_{\alpha(i)}$ for all $i=1,\dots,n$. Moreover, for any $i$ and $k$ such that $\alpha^k(i)=i$, the system $(X_i,\phi^k)$ is a mixing Smale space.
\end{theo}

Since the stable and unstable orbits are the same for $(X,\phi)$ and $(X,\phi^k)$, it is a simple matter to see that, applying the previous theorem to $(\Sigma,\sigma)$, we get a decomposition
\begin{equation}\label{eq:ds}
D^s(\Sigma) \cong D^s(\Sigma_1)\oplus\cdots\oplus D^s(\Sigma_n),
\end{equation}
(see also \cite[Section 2]{put:funct}).

\begin{rema}
In this paper we consider the dimension group merely as a group-invariant, without keeping track of the positive cone and of the induced automorphism (for more details, see \cite[Chapter 7]{lind:marcus}). Since the decomposition in \eqref{eq:ds} holds at the level of $C^*$-algebras, the positive cones decompose along the same shape. The induced automorphism (which also exists at the $C^*$-level) permutes the summands according to $\alpha$ as in Theorem \ref{thm:decomp}.
\end{rema}

In view of the preceding discussion, for the rest of this subsection we assume that $(\Sigma,\sigma)$ is mixing, in particular the global stable sets are dense.

\begin{lemm}
Let $f\in\Sigma$ and define $CO^s_{f}(\Sigma,\sigma)=\{ E \in CO^s_\epsilon(\Sigma,\sigma) \mid E \subseteq \Sigma^s(f)\}$. Consider the abelian group $D^s_f(\Sigma,\sigma)$, defined as in Definition \ref{def:dimgrp}, but replacing $CO^s(\Sigma,\sigma)$ with $CO^s_f(\Sigma,\sigma)$. Then we have $D^s(\Sigma,\sigma)\cong D^s_f(\Sigma,\sigma)$.
\end{lemm}
\begin{proof}
Given $E\in CO^s_\epsilon(\Sigma,\sigma)$, it is sufficient to prove $[E]=[F]$ for some $F\in CO^s_f(\Sigma,\sigma)$. Suppose $E\subseteq \Sigma^s(e,\epsilon)$ and let $\Sigma(e,\epsilon)$ denote the open ball centered at $e$ of radius $\epsilon$. Note that $\Sigma^s(e,\epsilon)\subseteq \Sigma(e,\epsilon)$. Since $\Sigma^s(f)$ is dense, we can find $f^\prime \in \Sigma^s(f)\cap \Sigma(e,\epsilon)$ and define $F=[f^\prime,E]$. The basic properties of the bracket imply $[E,F]=E,[F,E]=F$.
\end{proof}

\begin{rema}
It is clear that $CO^s_{f}(\Sigma,\sigma)$ gives a basis for the topology of $\Sigma^s(f)$. Thus $D^s(\Sigma)$ is generated by equivalence classes of basic clopens in some global stable~set.
\end{rema}

Let $R^u(\Sigma,f)$ be the set of pairs of unstably equivalent points which belong to the stable orbit through $f\in \Sigma$. This is an amenable, étale groupoid when endowed with the topology as in \cite[Section 1.2]{thomsen:smale} (see also \cite[Theorem 3.6]{put:spiel}).

\begin{rema}\label{rem:amen}
In \cite[page 14]{put:algSmale} the question arises if stable and unstable equivalence relations of any mixing Smale space are locally compact amenable groupoids. The answer is positive and the proof is as follows: by \cite[Corollary 3.8]{deeley:strung}, in the equivalence class (in the sense of \cite{murewi:morita}) of such equivalence relations we can find étale amenable groupoids, because their corresponding $C^*$-algebras have finite nuclear dimension (see \cite[Theorem 5.6.18]{brown:ozawa}). Amenability is invariant under this sort of equivalence by \cite[Theorem 2.2.17]{renroch:amgrp}.
\end{rema}

A subbase for the topology on $R^u(\Sigma,f)$ is given by triples $(E,F,\gamma)$ where $E,F$ are basic clopens of the unit space $\Sigma^s(f)$ and $\gamma\colon E\to F$ is homeomorphism such that $(e,\gamma(e))\in R^u(\Sigma,f)$ for all $e\in E$.
We consider the following ``categorification'' of $R^u(\Sigma,f)$: define a category $\Cc(\Sigma,f)$ whose objects are the clopens in $CO^s_{f}(\Sigma,\sigma)$ and morphisms $E\to F$ are inclusions $E\hookrightarrow F$ and triples $(E,F,\gamma)$ as above.

Recall that the $K$-theory $K_0(\Cc)$ of an additive category $(\Cc,\oplus)$ is the abelian group generated by isomorphism classes $[E]$ of objects $E\in \Cc$ subject to the relation $[E\oplus F]=E+F$. If we interpret isomorphism classes as $(E,F,\gamma)$-orbits in $\Cc(\Sigma,f)$, and we take $E\oplus F$ to mean $E\cup F,E\cap F=\emptyset$, then we obtain a well-defined abelian group $K_0(\Cc(\Sigma,f))$.

\begin{rema}
Note the condition $E\oplus F$ is completely determined by inclusions. Indeed unions and intersections are specific colimits and limits in $\Cc(\Sigma,f)$.
\end{rema}

\begin{theo}\label{thm:dimCgroup}
We have $D^s(\Sigma)\cong K_0(\Cc(\Sigma,f))$ for any $f\in\Sigma$.
\end{theo}
\begin{proof}
Let us take $E$ and $F$ such that $[E]=[F]$. In particular there is $n\in\N$ and $f\in \sigma^{-n}(F)$ such that $[f,\sigma^{-n}(E)]=\sigma^{-n}(F)$. It is easy to see that the map 
\begin{equation}\label{eq:gamma}
\gamma(e)=\sigma^{n}([f,\sigma^{-n}(e)])
\end{equation}
is a homeomorphism of $E$ onto $F$, and obvioulsy $\sigma^{-n}(e)$ belongs to the local unstable set of $\sigma^{-n}(\gamma(e))$, therefore $(e,\gamma(e))\in R^u(\Sigma,f)$. 

Conversely, if $(E,F,\gamma)$ is an isomorphism in $\Cc(\Sigma,f)$, then by \cite[Lemma 4.14]{thomsen:smale} (and compactness), we can partitition $E$ in a finite number of clopens $E_1,\dots,E_n$, and correspondingly $F$ in $F_1,\dots,F_n$, where $E_i$ is homeomorphic to $F_i$ through a map in the form of \eqref{eq:gamma}. Therefore $[E_i]=[F_i]$ and by the defining relation $[E]=[F]$.
\end{proof}

If $\chi_E$ is the indicator function of the clopen $E$ inside the groupoid $C^*$-algebra $C^*(R^u(\Sigma,f))$, then a little thinking over the assignment $E\mapsto \chi_E$ gives the following well-known result (for more details see \cite[Section 4.3]{thomsen:smale}).

\begin{coro}\label{cor:dimcalg}
There is an isomorphism 
\[
K_0(C^*(R^u(\Sigma,f)))\cong K_0(\Cc(\Sigma,f)) \cong D^s(\Sigma)
\]
for any $f\in \Sigma$.
\end{coro}

\begin{rema}
As was already implicitly noted in Remark \ref{rem:amen}, the reason why the choice of $f\in \Sigma$ doesn't affect the $K$-theory group is to be found in the statement that reducing a groupoid to a transversal preserves its equivalence class, as explained in more detail in \cite[Example 2.7]{murewi:morita}.
\end{rema}

\subsection{Complexes}

The maps in \eqref{eq:deltalm} will induce group morphisms denoted respectively $\delta_l^s,\delta_{,m}^{s*}$.
For each $L,M\geq 0$, we consider maps
\begin{align}\label{eq:cc}
\partial^s_{L,M}\colon D^s(\Sigma_{L,M}(\pi))&\to D^s(\Sigma_{L-1,M}(\pi))\\\notag
\partial^s_{L,M}&=\sum_{0\leq l\leq L} (-1)^{l}\delta_l^s\\\label{eq:coc}
\partial^{s*}_{L,M}\colon D^s(\Sigma_{L,M}(\pi))&\to D^s(\Sigma_{L,M+1}(\pi))\\\notag
\partial^s_{L,M}&=\sum_{0\leq m\leq M+1} (-1)^{L+m}\delta_{,m}^{s*}.
\end{align}
It is clear from the definition that
\[
\partial^s_{L,M+1}\circ\partial^{s*}_{L,M}=\partial^{s*}_{L-1,M}\circ\partial^s_{L,M}.
\]
Furthermore, by applying \cite[Theorems 2.6.11, 2.6.12, 4.1.14]{put:HoSmale}, we have that
\begin{itemize}
\item for each  $M\geq 0$, \eqref{eq:cc} is a chain complex;  
\item for each $L\geq 0$, \eqref{eq:coc} is a cochain complex.
\end{itemize}

Altogether, we have a double complex $(C^s(\pi)_{\bullet,\bullet},\partial^s,\partial^{s*})$, where 
\begin{align*}
C^s(\pi)_{L,M}=
\begin{cases}
D^s(\Sigma_{L,M}(\pi)) & \text{if $L\geq 0$ and $M\geq 0$}\\
0 & \text{else}.
\end{cases}
\end{align*}

The \emph{totalization} of this complex is the chain complex $(\Tot(C^s(\pi))_\bullet,d^s)$, where
\begin{align*}
\Tot(C^s(\pi))_N&=\bigoplus_{L-M=N} C^s(\pi)_{L,M}\\
d^s_{L,M}&=\partial^s_{L,M}+\partial^{s*}_{L,M}\\
d^s_N=\bigoplus_{L-M=N}d^s_{L,M}\colon &\Tot(C^s(\pi))_N\to \Tot(C^s(\pi))_{N-1}.
\end{align*}

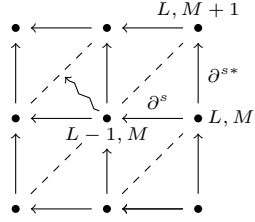
\begin{figure}
\centering
\begin{tikzpicture}[font=\scriptsize]
\fill[black] (24mm,24mm) circle[radius=1.5pt] node[above] {$L,M+1$};
\fill[black] (24mm,0mm) circle[radius=1.5pt] node[above] {};
\fill[black] (24mm,12mm) circle[radius=1.5pt] node[right] {$L,M$};
\fill[black] (12mm,24mm) circle[radius=1.5pt] node[above] {};
\draw[->] (2.4,1.4)-- (2.4,2.2) node[anchor=north west, yshift=-2mm] {$\partial^{s*}$};
\draw[->] (2.4,0.2)-- (2.4,1.0) node[anchor=north west, yshift=-2mm] {};
\draw[->] (0,0.2)-- (0,1.0) node[anchor=north west, yshift=-2mm] {};
\draw[->] (2.2,1.2)-- (1.4,1.2) node[anchor= south west, xshift=2mm] {$\partial^s$};
\draw[->] (1.0,2.4)-- (0.2,2.4) node[anchor= south west, xshift=2mm] {};
\draw[->] (2.2,2.4)-- (1.4,2.4) node[anchor= south west, xshift=2mm] {};
\draw[->] (1.0,1.2)-- (0.2,1.2) node[anchor= south west, xshift=2mm] {};
\draw[->] (1.2,1.4)-- (1.2,2.2) node[anchor=north west, yshift=-2mm] {};
\draw[->] (2.2,0)-- (1.4,0) node[anchor= south west, xshift=2mm] {};
\draw[->] (1.2,0.2)-- (1.2,0.8) node[anchor=north west, yshift=-2mm] {};
\draw[->] (2.2,0)-- (1.4,0) node[anchor= south west, xshift=2mm] {};
\draw[->] (0,1.4)-- (0,2.2) node[anchor=north west, yshift=-2mm] {};
\draw[->] (1,0)-- (0.2,0) node[anchor= south west, xshift=2mm] {};
\fill[black] (12mm,0mm) circle[radius=1.5pt] node[below] {};
\fill[black] (12mm,12mm) circle[radius=1.5pt] node[below] {$L-1,M$};
\fill[black] (0mm,24mm) circle[radius=1.5pt] node[left] {};
\draw [dashed] (0.2,0.2)-- (0.9,0.9) node[anchor=south east] {};
\draw [dashed] (1.4,1.4)-- (2.2,2.2) node[anchor=south east] {};
\draw [dashed] (0.2,1.4)-- (1.0,2.2) node[anchor=south east] {};
\draw [dashed] (1.4,0.2)-- (2.2,1.0) node[anchor=south east] {};
\draw [->,
line join=round,
decorate, decoration={
    zigzag,
    segment length=8,
    amplitude=.9,post=lineto,
    post length=2pt
}]  (1.1,1.3) -- (0.65,1.75);
\fill[black] (0mm,12mm) circle[radius=1.5pt] node[left] {};
\fill[black] (0mm,0mm) circle[radius=1.5pt] node[left] {};
\end{tikzpicture}
\caption[The complexes $C^s(\pi)_{\bullet,\bullet}$ and $\Tot(C^s(\pi))_\bullet$]{A representation of the complexes $C^s(\pi)_{\bullet,\bullet}$ and $\Tot(C^s(\pi))_\bullet$. The direct sums of the groups lying on the dashed diagonals give $\Tot(C^s(\pi))_\bullet$. The differentials $d^s$ (e.g., the zigzag arrow in the top-left square) run from south-east to north-west, decreasing degree by $1$.}
\end{figure}

By slightly modifying the invariant $\Sigma\mapsto D^s(\Sigma)$, we can introduce a cochain complex which is related to \eqref{eq:coc}, and will give rise to another double complex. We summarize the details of this construction (see \cite[Definition 4.1.5]{put:HoSmale}):
\begin{itemize}
\item For any $L\geq 0$, the symmetric group $S_{M+1}$ acts by automorphisms (in particular, $s$-bijective maps) on $\Sigma_{L,M}(\pi)$. Define the group
\[
D^s_{,\A}(\Sigma_{L,M}(\pi))=\{ a\in D^s(\Sigma_{L,M}(\pi)) \mid  a=\text{sgn}(\beta)\beta(a)\text{ for all $\beta\in S_{M+1}$}\};
\]

\item By \cite[Lemma 5.1.6]{put:HoSmale}, we have
\begin{align*}
\partial^s_{L,M}D^s_{,\A}(\Sigma_{L,M}(\pi))&\subseteq D^s_{,\A}(\Sigma_{L-1,M}(\pi))\\
\partial^{s*}_{L,M}D^s_{,\A}(\Sigma_{L,M}(\pi))&\subseteq D^s_{,\A}(\Sigma_{L,M+1}(\pi));
\end{align*}
\item Define a bicomplex $(C_{,\A}^s(\pi)_{\bullet,\bullet},\partial^s,\partial^{s*})$ by setting
\[
C_{,\A}^s(\pi)_{L,M}=D^s_{,\A}(\Sigma_{L,M}(\pi));
\]
\item The inclusion map $J\colon D^s_{,\A}(\Sigma_{L,M}(\pi)) \to D^s(\Sigma_{L,M}(\pi))$ induces chain maps
\begin{align*}
(C_{,\A}^s(\pi)_{\bullet,\bullet},\partial^s,\partial^{s*}) &\to (C^s(\pi)_{\bullet,\bullet},\partial^s,\partial^{s*})\\
(\Tot(C_{,\A}^s(\pi))_\bullet,d^s)&\to (\Tot(C^s(\pi))_\bullet,d^s),
\end{align*}
and in particular, for each $L\geq 0$, a cochain map
\[
(C_{,\A}^s(\pi)_{L,\bullet},\partial^{s*}) \to (C^s(\pi)_{L,\bullet},\partial^{s*}).
\]
\end{itemize}

The advantage in using the complex just defined lies in the following propositions, proved in \cite[Theorem 4.2.12, Theorem 4.3.1]{put:HoSmale}
\begin{prop}\label{prop:boundM}
There is $N\geq 0$ such that $C_{,\A}^s(\pi)_{L,M}=0$ whenever $M\geq N$.
\end{prop}
\begin{prop}\label{prop:isochain}
For each, $L\geq 0$, the cochain map
\[
J\colon (C_{,\A}^s(\pi)_{L,\bullet},\partial^{s*}) \to (C^s(\pi)_{L,\bullet},\partial^{s*})
\]
is a quasi-isomorphism, i.e., for each $N\in \Z$ there are induced isomorphisms
\[
J_*\colon H_N(C_{,\A}^s(\pi)_{L,\bullet},\partial^{s*})\cong H_N(C^s(\pi)_{L,\bullet},\partial^{s*}).
\]
\end{prop}

\begin{defi}
By considering the symmetric group action $S_{L+1}$ on $\Sigma_{L,M}(\pi)$, one can introduce another invariant
\begin{equation*}
D^s_{\mathcal{Q}}(\Sigma_{L,M}(\pi))=\frac{D^s(\Sigma_{L,M}(\pi))}{D^s_{\mathcal{B}}(\Sigma_{L,M}(\pi))},
\end{equation*}
where $D^s_{\mathcal{B}}(\Sigma_{L,M}(\pi))$ is the subgroup of $D^s(\Sigma)$ generated by
\begin{itemize}
\item all elements $a$ satisfying $\alpha(a)=a$ for some non-trivial transposition $\alpha$ in the symmetric group $S_{L+1}$;
\item all elements of the form $a-\text{sgn}(\alpha)\alpha(a)$, where $\alpha\in S_{L+1}$.
\end{itemize}
Associated to this invariant is the bicomplex denoted $C_{\mathcal{Q}}^s(\pi)$ in \cite[Chapter 5]{put:HoSmale}. This complex enjoys the analogous property of Proposition \ref{prop:boundM}, i.e., it is zero outside a bounded region in the $L$-direction.

By combining both approaches, one can also introduce a fourth bicomplex, denoted $C_{\mathcal{Q},\A}^s(\pi)$ and based on the following ``dimension group'':
\[
D^s_{\mathcal{Q,\A}}(\Sigma_{L,M}(\pi))=\frac{D^s_{\mathcal{,\A}}(\Sigma_{L,M}(\pi))}{D^s_{\mathcal{B}}(\Sigma_{L,M}(\pi))\cap D^s_{\mathcal{,\A}}(\Sigma_{L,M}(\pi))}.
\]
This complex is zero outside a bounded rectangle of the first quadrant. Further details on these constructions are found in \cite[Definition 5.1.7]{put:HoSmale}. It is proved in \cite[Section 5.3]{put:HoSmale} that there are quasi-isomorphisms
\[
C_{,\A}^s(\pi)\to C_{\mathcal{Q},\A}^s(\pi)\to C_{\mathcal{Q}}^s(\pi).
\]
These results and constructions will be obtained through different methods in the next sections of this paper.

By definition, the (stable) \emph{homology groups of $(X,\phi)$} are given by ($N\in \Z$)
\[
H_N^s(X,\phi)=H_N(\Tot(C^s_{\mathcal{Q},\A}(\pi))_\bullet,d^s).
\]
It is proved in \cite[Section 5.5]{put:HoSmale} that this definition does not depend on the particular choice of $s/u$-bijective map $\pi$.
\end{defi}

\section{\texorpdfstring{$C_{,\A}^s(\pi)$ is quasi-isomorphic to $C^s(\pi)$}{Quasi-isomorphic bicomplexes}}\label{sec:ccc}

We are going to prove in this section that $C_{,\A}^s(\pi)$ is quasi-isomorphic to $C^s(\pi)$. This is the missing (but conjectured) result from Putnam's memoir \cite[page 90]{put:HoSmale}. For brevity, we write $C=C^s(\pi)$ and $C_\A=C^s(\pi)_{,\A}$.

There are at least two reasons why this quasi-isomorphism is important: firstly, it is clear that $C$ is the most straight-forward among the definable complexes for the homology of Smale spaces, and therefore it is a basic fundamental result that it's computing the same invariants as the other complexes. Secondly, and maybe more importantly, $C$ is also the complex with the most evident connection to $K$-theory for the associated $C^*$-algebras. Indeed, if we consider the $C^*$-morphisms induced by $\delta_{l}$ and $\delta_{,m}$ as explained in \cite{put:funct}, their corresponding $K$-theory maps agree with $\delta_l^s$ and $\delta_{,m}^{s*}$ after the identification given in Corollary \ref{cor:dimcalg}.

\subsection{Filtrations}
We proceed by defining the \emph{vertical filtration} on $C_{\bullet,\bullet}$, i.e., the family of subcomplexes given by ($p\in \Z$)
\begin{equation*}
F_pC_{L,M}=
\begin{cases}
C_{L,M} & \text{if $L\leq p$}\\
0 & \text{else},
\end{cases}
\end{equation*}
in other words everything to the right of the vertical line $L=p$ is set to zero, see Figure \ref{fig:filtr}.

\begin{figure}
\centering
\begin{tikzpicture}[font=\scriptsize]
\draw (24mm,24mm) node {0};
\draw (24mm,12mm) node {0};
\draw (24mm,0mm) node {0};
\fill[black] (12mm,24mm) circle[radius=1.5pt] node[above] {};
\draw[->] (0,0.2)-- (0,1.0) node[anchor=north east, yshift=-2mm] {$\partial^{s*}$};
\draw[->] (1.0,2.4)-- (0.2,2.4) node[anchor= south west, xshift=2mm] {};
\draw[->] (1.0,1.2)-- (0.2,1.2) node[anchor= south west, xshift=2mm] {};
\draw[->] (1.2,1.4)-- (1.2,2.2) node[anchor=north west, yshift=-2mm] {};
\draw[->] (1.2,0.2)-- (1.2,0.8) node[anchor=north west, yshift=-2mm] {};
\draw[->] (0,1.4)-- (0,2.2) node[anchor=north west, yshift=-2mm] {};
\draw[->] (1,0)-- (0.2,0) node[anchor= south west, xshift=2mm] {$\partial^s$};
\fill[black] (12mm,0mm) circle[radius=1.5pt] node[below] {$p,M$};
\fill[black] (12mm,12mm) circle[radius=1.5pt] node[below] {$p,M\hspace*{-0.4em}+\hspace*{-0.4em} 1$};
\fill[black] (0mm,24mm) circle[radius=1.5pt] node[left] {};
\draw [dashed] (0.2,0.2)-- (0.82,0.82) node[anchor=south east] {};
\draw [dashed] (1.4,1.4)-- (2.2,2.2) node[anchor=south east] {};
\draw [dashed] (0.2,1.4)-- (1.0,2.2) node[anchor=south east] {};
\draw [dashed] (1.4,0.2)-- (2.2,1.0) node[anchor=south east] {};
\draw [->,
line join=round,
decorate, decoration={
    zigzag,
    segment length=8,
    amplitude=.9,post=lineto,
    post length=2pt
}]  (1.1,1.3) -- (0.65,1.75);
\fill[black] (0mm,12mm) circle[radius=1.5pt] node[left] {};
\fill[black] (0mm,0mm) circle[radius=1.5pt] node[left] {};
\end{tikzpicture}
\caption[The vetical filtration]{The vetical filtration.}\label{fig:filtr}
\end{figure}
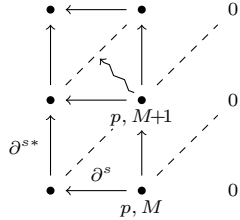

The resulting family $\{\Tot(F_pC)_\bullet\mid p\in \Z\}$ is a filtration of the totalization chain complex. Note there is a chain of inclusions
\begin{equation}\label{eq:filtr}
\cdots\subseteq\Tot(F_pC)_\bullet\subseteq\Tot(F_{p+1}C)_\bullet\subseteq\cdots.
\end{equation}

In complete analogy, we get a filtration
\begin{equation}\label{eq:filtrA}
\cdots\subseteq\Tot(F_pC_\A)_\bullet\subseteq\Tot(F_{p+1}C_\A)_\bullet\subseteq\cdots.
\end{equation}

The following remarks will be important in the next subsection.

\begin{rema}\label{rem:ex}
The filtration in \eqref{eq:filtr} is \emph{exhaustive}, i.e., the union over all $p$ of $\Tot(F_pC)_\bullet$ is $\Tot(C)_\bullet$. Note this implies the induced filtration on homology is also exhaustive. The same holds for \eqref{eq:filtrA}.
\end{rema}
\begin{rema}\label{rem:bb}
The filtration in \eqref{eq:filtr} is \emph{bounded below}, i.e., for each $N\in \Z$ there exists $s\in\Z$ such that $\Tot(F_sC)_N=0$. For $N\geq 0$ we can take $s=N-1$; when $N<0$ we take $s=-1$. Note this implies the induced filtration on homology is also bounded below. The same holds for \eqref{eq:filtrA}.
\end{rema}

\subsection{Spectral sequences}
A filtration of a chain complex gives rise to a spectral sequence, see \cite[Theorem 5.4.1]{wei:homalg} for a proof.

\begin{prop}\label{prop:spseq}
A filtration $\F$ of a chain complex $\Cc$ determines a spectral sequence:
\begin{align*}
E^0_{pq}&=F_p\Cc_{p+q}/F_{p-1}\Cc_{p+q}\\
E^1_{pq}&=H_{p+q}(E^0_{p\bullet}).
\end{align*}
\end{prop}

In order to discuss convergence for the spectral sequence in Proposition \ref{prop:spseq}, we introduce a bit of terminology (we follow \cite[Chapter 5]{wei:homalg}). The expert reader may skip to Corollary \ref{cor:conv}.

Recall that a (homology) spectral sequence is \emph{bounded below} if for each $n$ there is $s=s(n)$ such that the terms $E^r_{pq}$ with $p+q=n$ vanish for all $p<s$. A spectral sequence is \emph{regular} if for each $p$ and $q$ the differentials $d^r_{pq}$ leaving $E^r_{pq}$ are zero for all large $r$. 

\begin{rema}\label{rem:bbcase}
Bounded below spectral sequences are regular. If $\F$ is a bounded below filtration (see Remark \ref{rem:bb}), then the spectral sequence in Proposition \ref{prop:spseq} is bounded below, hence regular.
\end{rema}

For each $n\in \Z$, the homology group $H_n(\Cc)$ receives an induced filtration 
\[
\cdots\subseteq \F_pH_n(\Cc)\subseteq \F_{p+1}H_n(\Cc)\subseteq \cdots\subseteq H_n(C).
\]
We say the spectral sequence \emph{abuts to} to $H_*(\Cc)$ if, for all $p,q,n\in \Z$,
\begin{enumerate}
\item
there are isomorphisms 
\begin{equation}\label{eq:weakconv}
\beta_{pq}\colon E^\infty_{pq}\cong \F_p H_{p+q}(\Cc)/\F_{p-1} H_{p+q}(\Cc);
\end{equation}
\item $H_n(\Cc)=\cup_p \F_p H_n(\Cc)$;
\item $\cap \F_pH_n(\Cc)=0$.
\end{enumerate}
When $(\F,\Cc)=(F,\Tot(C))$ or $(\F,\Cc)=(F,\Tot(C_\A))$, items 2 and 3 above follow from Remarks \ref{rem:ex} and \ref{rem:bb} respectively.

We say the spectral sequence \emph{converges to} $H_*(\Cc)$ if it abuts to $H_*(\Cc)$, it is regular, and it holds for each $n\in \Z$ that
\[
H_n(\Cc)=\varprojlim_{p\in \Z} \frac{H_n(\Cc)}{\F_pH_n(\Cc)}.
\]
Note that a bounded below (hence regular) spectral sequence always satisfies the condition above, therefore it converges to $H_*(\Cc)$ as soon as the abutment condition holds. This applies to the spectral sequences associated to $(F,\Tot(C))$ and $(F,\Tot(C_\A))$, because of Remarks \ref{rem:bb} and \ref{rem:bbcase}.

Suppose $\{E^r_{pq}\}$ and $\{E^{\prime r}_{pq}\}$ satisfy \eqref{eq:weakconv} with respect to $H_*$ and $H^\prime_*$ respectively. We say that a map $h\colon H_*\to H^\prime_*$ is \emph{compatible} with a morphism $f\colon E\to E^\prime$ if
\begin{itemize}
\item $h(\F_p H_n)\subseteq \F_pH_n^\prime$ for all $n\in \Z$;
\item the induced maps $\F_p H_n/\F_{p-1} H_n\to\F_p H^\prime_n/\F_{p-1}H_n^\prime$ correspond under $\beta$ and $\beta^\prime$ to $f^\infty_{pq}\colon E^\infty_{pq}\to E^{\prime\infty}_{pq}$, $q=n-p$.
\end{itemize}

We recall the following result \cite[Theorem 5.5.1]{wei:homalg}.

\begin{theo}\label{thm:convergence}
Condition \eqref{eq:weakconv} holds for bounded below spectral sequences. 

In particular, if $(\F,\Cc)$ is a filtered chain complex where $\F$ is exhaustive and bounded below, then the associated spectral sequence is bounded below and converges to $H_*(\Cc)$. 
Moreover, the convergence is natural: if $f\colon\Cc\to\Cc^\prime$ is a map of filtered complexes, then the map $f_*\colon H_*(\Cc)\to H_*(\Cc^\prime)$ is compatible with the corresponding morphism of spectral sequences.
\end{theo}
\begin{coro}\label{cor:comparison}
With notations as above, if $f^r\colon E^r_{pq}\cong E^{\prime r}_{pq}$ is an isomorphism for all $p,q$ and some $r$ (hence for $r=\infty$, see \cite[Lemma 5.2.4]{wei:homalg}), then $f_*\colon H_*(\Cc)\to H_*(\Cc^\prime)$ is an isomorphism.
\end{coro}

\begin{coro}\label{cor:conv}
There are convergent spectral sequences
\begin{align*}
E^1_{pq}=H_q((C_\A)_{p,\bullet},\partial^{s*})&\Rightarrow H_{p+q}(\Tot(C_\A)_\bullet,d^s)\cong H_{p+q}^s(X,\phi)\\
E^{\prime 1}_{pq}=H_q(C_{p,\bullet},\partial^{s*})&\Rightarrow H_{p+q}(\Tot(C)_\bullet,d^s).
\end{align*}
Furthermore, the inclusion chain map
\[
J\colon (\Tot(C{_\A})_\bullet,d^s)\to (\Tot(C)_\bullet,d^s)
\]
is a quasi-isomorphism.
\end{coro}
\begin{proof}
The spectral sequences arise by applying Proposition \ref{prop:spseq} with $(\F,\Cc)=(F,\Tot(C))$ and $(\F,\Cc)=(F,\Tot(C_\A))$. Convergence follows from Theorem \ref{thm:convergence}. The map $J$ induces isomorphisms
\[
J^1\colon E^1_{pq}=H_q((C_\A)_{p,\bullet},\partial^{s*})\cong E^{\prime 1}_{pq}=H_q(C_{p,\bullet},\partial^{s*})
\]
for all $p,q$ by Proposition \ref{prop:isochain}. The result follows from Corollary \ref{cor:comparison} above.
\end{proof}
\begin{coro}
Homology groups for the Smale space $(X,\phi)$ can be equally defined as ($N\in \Z$)
\[
H_N^s(X,\phi)=H_N(\Tot(C^s(\pi))_\bullet,d^s).
\]
\end{coro}

\section{Simplicial viewpoint}\label{sec:sim}

The $s/u$-bijective pair $\pi=(Y,\psi,\pi_s,Z,\zeta,\pi_u)$ for $(X,\phi)$ gives rise to a \emph{bisimplicial} Smale space $(\Sigma_{L,M}(\pi))_{L,M\geq 0}$. We will drop the reference to $\pi$ for brevity.

The face maps are given by \eqref{eq:deltalm}. We stress that the $\delta_l$'s are $s$-bijective and the $\delta_{,m}$'s are $u$-bijective. The degeneracy maps are as follows:
\begin{align*}
s_l\colon  \Sigma_{L,M}&\to \Sigma_{L+1,M}\\
(y_0,\dots,y_l,\dots,y_L,z_0,\dots,z_M)&\mapsto (y_0,\dots,y_l,y_l\dots,y_L,z_0,\dots,z_M)\\
s_{,m}\colon  \Sigma_{L,M}&\to \Sigma_{L,M+1}\\
(y_0,\dots,y_L,z_0,\dots,z_m,\dots,z_M)&\mapsto (y_0,\dots,y_L,z_0,\dots,z_m,z_m,\dots,z_M),
\end{align*}
for $l=0,\cdots, L$ and $m=0,\dots, M$.

\begin{rema}
Let us point out a notational difference between the present paper and \cite{put:HoSmale}. Whenever we write $\delta_l$, Putnam appends an extra comma to the subscript, i.e., $\delta_{l,}$. We deviated from this convention because the distinction between $\delta_{l,}$ and $\delta_{,m}$ is less important in our case, for we rarely use $\delta_{,m}$. In accordance with this usage, we denote the map $s_l$ and \emph{not} $s_{l,}$.
\end{rema}

\begin{rema}
Note that $s_l(\Sigma_{L,M})\subseteq \Sigma_{L+1,M}$ is a closed shift-invariant system, clearly isomorphic to $\Sigma_{L,M}$. The same holds for $s_{,m}(\Sigma_{L,M})\subseteq \Sigma_{L,M+1}$.
\end{rema}

\begin{rema}
It is not difficult to see that, for each $l$, the map $s_l$ is $s$-bijective because its inverse is given by $\delta_l$. The situation is different for the maps $s_{,m}$: they are only $s$-resolving.
\end{rema}

\begin{prop}
There are induced maps
\begin{align*}
s_l^s\colon  D^s(\Sigma_{L,M})&\to D^s(\Sigma_{L+1,M})\\
s_{,m}^{s*}\colon  D^s(\Sigma_{L,M+1})&\to D^s(\Sigma_{L,M}).
\end{align*}
Moreover, the map $s^s_l$ is split-injective.
%
\end{prop}
\begin{proof}
Recall that (equivalence classes of) compact open sets inside stable orbits provide generators for the dimension groups. Given one of such classes $[E]\in D^s(\Sigma_{L,M})$, the assignment $[E]\mapsto [s_l(E)]$ is a well-defined group morphism because the map $s_l\colon \Sigma^s_{L,M}(e)\to \Sigma^s_{L+1,M}(s_l(e))$ is a homeomorphism.
The splitting for the map $s_l^s$ is given by $\delta_l^s$.
The definition of $s_{,m}^{s*}$ is given by $E\mapsto s_{,m}^{-1}(E)$. This preimage is compact because $s_{,m}$ is proper, as $s$-resolving maps are proper \cite[Theorem 2.5.4]{put:HoSmale}. Since $E\cap F=\emptyset$ implies $s_{,m}(E)\cap s_{,m}(F)=\emptyset$, the map respects the group operation. 
\end{proof}

The theorem below follows easily from the discussion so far (and some simple verifications). See \cite[Chapter 8]{wei:homalg} for the Dold-Kan correspondence.

\begin{theo}
Applying the $D^s$-functor to the bisimplicial space $\Sigma_{\bullet,\bullet}$ results in a simplicial cosimplicial group $(D^s(\Sigma_{\bullet,\bullet}),\delta_l^s,s_l^s,\delta^s_{,m},s^{s*}_{,m})$. Furthermore the \emph{unnormalized} double complex associated to said group via the Dold-Kan correspondence is $(C_{L,M},\partial^s,\partial^{s*})$, as defined in Section \ref{sec:prelims}.
\end{theo}

\begin{rema}
As was mentioned at the beginning of Section \ref{sec:ccc}, the complex $(C_{L,M},\partial^s,\partial^{s*})$ is also the result of applying the $K$-theory functor to $\Sigma_{\bullet,\bullet}$. The intermediate step in this case is constructing the associated $C^*$-algebras, which are AF \cite[Section 4.3]{thomsen:smale}, so the odd $K$-groups vanish. 
\end{rema}

By considering the normalizations (sometimes called the \emph{Moore complexes}) associated to $D^s(\Sigma_{\bullet,\bullet})$ we obtain simplicial versions of the bicomplexes $C_\A,C_{\mathcal{Q}},C_{\A,\mathcal{Q}}$ that were previously introduced. It is well-known that these all yield isomorphic homology groups (see \cite[Theorem 8.3.8]{wei:homalg}). However, it should be noted that these complexes are not as useful as their ``symmetric'' counterpart (to be introduced in the next section), because they don't allow for computational simplifications as in Proposition \ref{prop:boundM}.

\subsection{Symmetric simplicial groups}

Fix $M\geq 0$ and consider the simplicial group $(\Sigma_{\bullet,M},\delta_l^s,s_l^s)$. It carries an action of the symmetric group $S_{L+1}$. Recall that this group is generated by the \emph{adjacent} transpositions $t_l=(l\;l+1)$, $l=0,\dots,L-1$ (see \cite[Section 5, Theorem 3]{johnson:groups}).

The functorial properties of the $D^s$-invariant easily give the theorem below. The notion of symmetric simplicial group is inspired by \cite{grandis:symsets}.

\begin{theo}\label{thm:sym}
The simplicial group $(\Sigma_{\bullet,M},\delta_l^s,s_l^s)$ is a \emph{symmetric object}, i.e., it carries an action of the transpositions $t_i$'s, subject to the defining relations of $S_{L+1}$ and to the following \emph{mixed relations}:
\begin{align*}
\delta^s_j t_i&=t_i \delta^s_j & s_j^st_i&=t_is_j^s & (i<j-1)\\
\delta^s_it_i&=\delta^s_{i+1} &   s_i^st_i&=t_{i+1}t_is^s_{i+1}\\
\delta^s_jt_i&=t_{i-1}\delta^s_j & s^s_jt_i&=t_{i+1}s^s_j & (i>j)\\
&&t_is^s_i&=s^s_i.
\end{align*}
\end{theo}

For some $l$ and $j=1,\dots,L+1-l$, we are going to consider the cycle $\sigma_j=(l+j\;l+j-1\;\cdots\;l+1)$ in $S_{L+1}$ and the compositions $\sigma_js_l$. Note $\sigma_1s_l=s_l$. In other words $\sigma_j s_l$ is an additional degeneracy map which repeats entry $y_l$ at coordinate $l+j$:
\[
\xymatrix{(y_0,\dots,y_l,\dots,y_L)\ar[r]^-{\sigma_js_l} & (y_0,\dots,y_l,\dots,y_{l+j-1},y_l,y_{l+j},\dots y_L)\in \Sigma_{L+1,M}.}
\]
As composition of $s$-bijective maps, the $\sigma_js_l$'s induce group morphisms
\[
\xymatrix{D^s(\Sigma_{L,M}) \ar[r]^-{(\sigma_js_l)^s} & D^s(\Sigma_{L+1,M}).}
\]

It is then natural to define the groups of \emph{degenerate} chains,
\[
\tilde{D}C_{L,M}=\sum_{l,j}(\sigma_j s_l)^s(C_{L-1,M}).
\] 
The subgroup $\sum_{l}(\sigma_1 s_l)^s(C_{L-1,M})$ is preserved by the differential $\partial^s$ thanks to the simplicial identities, but when $j>1$ the identities in Theorem \ref{thm:sym} give the following relation:
\begin{equation}\label{eq:remdpres}
\partial^s(\tilde{D} C_{L,M})\subseteq \tilde{D}C_{L-1,M} + \langle \sigma_j(a)-\text{sgn}(\sigma_j)(a)\mid a\in C_{L,M} \rangle,
\end{equation}
because $\delta^s_l(\sigma_js_l)^s(a)=\sigma_j(a)$ and $\delta^s_{l+j}(\sigma_js_l)^s(a)=a$.

\begin{lemm}
There is an equality ($a\in C_{L,M}$)
\begin{align*}
\langle \sigma_j(a)-\text{sgn}(\sigma_j)(a) \rangle&=\langle t_i(a)+a \mid i=1,\dots,L-1\rangle\\
&=\langle \alpha(a)-\text{sgn}(\alpha)(a) \mid \alpha\in S_{L+1}\rangle.
\end{align*}
\end{lemm}
\begin{proof}
Since the $t_i$'s are generators we can write $\alpha(a)=t_{i_1}\cdots t_{i_n}(a)$. Then we have
\begin{align*}
&(t_{i_1}\cdots t_{i_n}(a)+t_{i_2}\cdots t_{i_n}(a))-(t_{i_2}\cdots t_{i_n}(a)+t_{i_3}\cdots t_{i_n}(a))\\
{}+{}&(t_{i_3}\cdots t_{i_n}(a)+t_{i_4}\cdots t_{i_n}(a))-\cdots\pm(t_{i_n}(a)+a)=\alpha(a)\pm a.
\end{align*}
The sign is positive when $n$ is odd and negative when $n$ is even, i.e., it is in accordance with $-\text{sgn}(\alpha)$. Note that our notation for $\sigma_j$ does not make reference to the index $l$, so that $\sigma_j(a)$ for $j=2$ includes all elements of the form $t_i(a)$.
\end{proof}

We can now ``correct'' our definition of degenerate chains by setting $DC_{L,M}$ to be the group generated by $\tilde{D}C_{L,M}$ and $\langle \alpha(a)-\text{sgn}(\alpha)(a) \mid a\in D^s(\Sigma_{L,M}),\alpha\in S_{L+1}\rangle$. 

\begin{lemm}
$(DC_{\bullet,M},\partial^s)$ is a well-defined subcomplex of $(C_{\bullet,M},\partial^s)$.
\end{lemm}
\begin{proof}
In view of the remark in \eqref{eq:remdpres}, we only need to check what happens to $\partial^s(t_i(a)+a)$. By looking at the identities in Theorem \ref{thm:sym}, we see that it suffices to check the expression $\delta^s_i(t_i(a)+a)-\delta^s_{i+1}(t_i(a)+a)$. It is easy to see that $\delta^s_{i+1}t_i=\delta^s_i$ so we get
\[
\delta^s_i(t_i(a)+a)-\delta^s_{i+1}(t_i(a)+a)=\delta^s_{i+1}(a)+\delta^s_i(a)-\delta^s_i(a)-\delta^s_{i+1}(a)=0.
\]
Therefore $\partial^s$ preserves the subgroup $\langle \alpha(a)-\text{sgn}(\alpha)(a) \mid a\in D^s(\Sigma_{L,M}),\alpha\in S_{L+1}\rangle$.
\end{proof}

\begin{theo}\label{thm:degthm}
Consider the short exact sequence
\[
\xymatrix{0\ar[r]& DC_{\bullet,M} \ar[r] & C_{\bullet,M} \ar[r] & \frac{C_{\bullet,M}}{DC_{\bullet,M}} \ar[r] & 0.}
\]
The complex $DC_{\bullet,M}$ is acyclic, hence the projection map is a quasi-isomorphism. \end{theo}
\begin{proof}
Set $DC_L=DC_{L,M}$ for brevity. We filter $DC_{\bullet,M}$ by setting $F_{0}DC_L=0$ and
\begin{align*}
F_pDC_L&=\sum_{l=0}^{k}\sum_{j=1}^{L-l}(\sigma_js_l)^s(C_{L-1,M})\\
&+\sum_{j=1}^{n}(\sigma_js_{k+1})^s(C_{L-1,M})+\langle \alpha(a)-\text{sgn}(\alpha)(a)\rangle
\end{align*}
when $p=L+(L-1)+\cdots+(L-k)+n$ and $0\leq n\leq L-k-1$. When $p\geq L(L+1)/2$ we have $F_pDC_L=DC_L$.
The simplicial (and mixed) identities show that each $F_pDC_\bullet$ is a subcomplex. This filtration $F$ is bounded, so there is a convergent spectral sequence (see \cite[Theorem 5.5.1]{wei:homalg})
\[
E^1_{pq}=H_{p+q}(F_pDC_\bullet/F_{p-1}DC_\bullet)\Rightarrow H_{p+q}(DC_\bullet).
\]
We have reduced ourselves to showing that each $F_pDC/F_{p-1}$ is acyclic. We take $x\in DC_{L-1}$ and compute in $F_pDC/F_{p-1}$:
\begin{align*}
\partial^s(\sigma_ns_{k+1})^s(x)&=\sum_{i=k+n+2}^L(-1)^i(\sigma_ns_{k+1})^s(\delta^s_{i-1})(x)
\\
\partial^s(\sigma_ns_{k+1})^s(\sigma_ns_{k+1})^s(x)&+(\sigma_ns_{k+1})^s\partial^s(\sigma_ns_{k+1})^s(x)\\
&=\sum_{i=k+n+2}^{L+1}(-1)^i(\sigma_ns_{k+1})^s(\delta^s_{i-1})(\sigma_ns_{k+1})^s(x)\\
&-\sum_{i=k+n+2}^L(-1)^i(\sigma_ns_{k+1})^s(\sigma_ns_{k+1})^s(\delta^s_{i-1})(x)\\
&=(-1)^{p}(\sigma_ns_{k+1})^s(x).
\end{align*}
Hence $\psi_L=(-1)^{p}(\sigma_ns_{k+1})^s$ is a chain contraction of the identity map which implies $F_pDC/F_{p-1}$ is acyclic.
\end{proof}
\begin{coro}
There is an isomorphism of double complexes:
\[
((C_{\mathcal{Q}})_{\bullet,\bullet}),\partial^s,\partial^{s*})\cong\Biggl(\frac{C_{\bullet,M}}{DC_{\bullet,M}},\partial^s,\partial^{s*}\Biggr).
\]
In particular, for each $M\geq 0$ there is a quasi-isomorphism of chain complexes
\begin{equation}\label{eq:chainver}
((C_{\bullet,M}),\partial^s)\to ((C_{\mathcal{Q}})_{\bullet,M},\partial^s).
\end{equation}
\end{coro}
\begin{proof}
All we need to do is identifying $D^s_{\mathcal{B}}(\Sigma_{L,M})$ with $DC_{L,M}$. Obviously
\[
\langle \alpha(a)-\text{sgn}(\alpha)(a)\rangle=\langle a-\text{sgn}(\alpha)\alpha(a)\rangle,
\]
and elements in the image of the degeneracy maps are clearly left invariant by some non-trivial transposition. Given $[E]\in D^s_{\mathcal{B}}(\Sigma_{L,M})$ such that $[E]=[\alpha(E)]$ for some transposition, we need to show $[E]=[F]$ for some clopen $F$ in the image of a degeneracy map. Now suppose $\alpha=(i\; i+k)$ and define $F$ to be $(\sigma_ks_i)\delta_{i+k}(E)$. Then the condition $[E,F]=E,[F,E]=F$ trivially holds separately on each coordinate $y_l$ with $l\neq i+k$, and when $l=i+k$ we can check the condition replacing $F$ by $\alpha(E)$, because the coordinate of index $i+k$ is pointwise the same in $F$ and $\alpha(E)$.
\end{proof}

Note that \eqref{eq:chainver} is a chain version of Proposition \ref{prop:isochain} and is proved in \cite[Theorem 4.3.1]{put:HoSmale}. We have used Proposition \ref{prop:isochain} in order to establish the quasi-isomorphism $(\Tot(C{_\A})_\bullet,d^s)\to (\Tot(C)_\bullet,d^s)$ given by inclusion. Dually, it is natural to seek a quasi-isomorphism $(\Tot(C)_\bullet,d^s)\to(\Tot(C_{\mathcal{Q}})_\bullet,d^s)$ induced by projection, which makes use of \eqref{eq:chainver}. Of course the strategy is completely similar to Corollary \ref{cor:conv}, but considering the \emph{horizontal filtration} instead of the vertical one. We skip the details.

\begin{coro}
The projection map in Theorem \ref{thm:degthm} induces a quasi-isomorphism
\[
(\Tot(C)_\bullet,d^s)\to(\Tot(C_{\mathcal{Q}})_\bullet,d^s).
\]
\end{coro}

To complete the picture, we also give the dual version of Proposition \ref{prop:boundM}.

\begin{prop}
There is $N\geq 0$ such that $C_{L,M}=DC_{L,M}$ whenever $L\geq N$. 

Therefore $(C_{\mathcal{Q}})_{L,M}=0$ whenever $L\geq N$.
\end{prop}
\begin{proof}
Recall that we can choose the $s/u$-bijective pair $\pi=(Y,\psi,\pi_s,Z,\zeta,\pi_u)$ so that $\pi_s$ is finite-to-one. Let $N-1$ be the maximum cardinality of a fiber of $\pi_s$ and $L\geq N$. A generic generator for $D^s(\Sigma_{L,M})$ is a compact open in some stable orbit $E\subseteq \Sigma^s_{L,M}(e)$. By the choice of $L$, there are $i$ and $k$ such that $e=(y_0,y_1,\dots,y_L,\dots)$ with $y_i=y_{i+k}$. Since $\delta_{i+k}\colon \Sigma^s_{L,M}(e)\to \Sigma^s_{L-1,M}(\delta_{i+k}(e))$ is a homeomorphism, we see that $E=(\sigma_ks_i)(E)$.
\end{proof}

\subsection{Symmetric cosimplicial groups}

As was hinted at the end of the previous section, the methods so far can be promptly dualized by considering the symmetric cosimplicial group $(\Sigma_{L,\bullet},\delta_{,m}^{s*},(\sigma_i s_{,m})^{s*})$ for fixed $L\geq 0$. 

We will omit most details since this is a standard argument, and simply outline how to define the cochain complex of degenerate chains, which is the essential object needed to define the relevant ``symmetric'' Moore complex. 

Where we used quotients in the previous section, we now have subgroups; moreover, by interpreting ``coinvariants'' to mean equivalence classes modulo $\langle a-\text{sgn}(\alpha)\alpha(a)\rangle$, we are led to consider the dual notion of ``invariants''. This brings to defining
\begin{align*}
CC_{L,M}=\{a\in C_{L,M} \mid\, & (\sigma_js_{,m})^{s*}(a)=0,\, a-\text{sgn}(\alpha)\alpha(a)=0\\
&\text{for all $m=0,\dots,M, j=1,\dots,M+1-m,\alpha\in S_M$}\},
\end{align*}
that is the invariant elements lying in intersection of kernels for all degeneracy maps.

By the (dual) argument of \ref{thm:degthm} one can proceed to show that the quotient complex $C_{L,\bullet}/CC_{L,\bullet}$ is acyclic, so the inclusion 
\[
CC_{L,\bullet}\to C_{L,\bullet}
\]
is a quasi-isomorphism. Finally, we mention that the projection map in Theorem \ref{thm:degthm} clearly induces a chain map 
\[
CC_{\bullet,M}\to \frac{CC_{\bullet,M}}{DC_{\bullet,M}\cap CC_{\bullet,M}},
\]
which is an isomorphism on homology. An explicit inverse for the map is constructed in \cite[page 98]{put:HoSmale}.

\section{Projective covers}\label{sec:proj}

Given a sheaf $S$ over a paracompact Hausdorff space $X$, sheaf cohomology $H^*(X,S)$ is computed from the complex $\Hom_X(\Z,I^\bullet)$, where $I^\bullet$ is an injective resolution of $S$. It is true, but perhaps less well-known, that the same calculation can be performed by means of the complex $\Hom_X(E^\bullet X,S)$, where $E^\bullet X$ is a semi-simplicial resolution of $X$ arising from a projective cover $E$ of $X$ (see \cite{dyck:proj}).

This alternative path to computing sheaf cohomology calls for an analogy with the homology theory for Smale spaces. Indeed, we have seen how the defining complex arises by applying Krieger's invariant to the bisimplicial space induced by a chosen $s/u$-bijective pair. So the role of the global section functor is played, in our context, by the dimension group construction for subshifts. 

The analogy is stronger when we start with a Smale space with totally disconnected stable sets. In this case, the homology is computed by the complex $(C_{\mathcal{Q}}(\pi)_{\bullet,0},\partial^s)$ and the $s/u$-bijective pair is reduced to a simple $s$-bijective map
\begin{equation*}
\pi\colon \Sigma_0\to X,
\end{equation*}
where $\Sigma_0$ is a subshift of finite type (see \cite[Section 7.2]{put:HoSmale}). Thus in this case the analogy calls for considering $\Sigma_0$ as a ``projective'' cover of $X$, together with its associated simplicial resolution $\Sigma_\bullet$ obtained by taking iterated fibered products over~$\pi$.

It should be noted that, while the usage of the term ``resolution'' is somewhat justified (since by definition $\Sigma_\bullet$ computes the ``right'' homology groups), the attribute \emph{projective} requires further reasons. This section contains a simple theorem in this direction.

In the category of compact Hausdorff spaces and continuous maps, a projective object is a space $E$ such that, whenever we are given $f\colon E\to A$ and $g\colon B \twoheadrightarrow A$ (onto), there is $h\colon E\to B$ with $f=g\circ h$. 

A \emph{projective cover} of $X$ is a pair $(E,e)$ with $E$ projective and $e\colon E\twoheadrightarrow X$ \emph{irreducible}, i.e., mapping proper closed sets onto proper subsets.

Gleason \cite{glea:proj} has proved that projective covers exist and are unique (up to a homeomorphism making the obvious diagram commute). Moreover he showed that a space is projective if and only if it is extremally disconnected, i.e., the closure of each of its open sets is open. Recall that $\Sigma_0$ is a compact, Hausdorff, totally disconnected space. In general extremally disconnected Hausdorff spaces are totally disconnected, but the converse does not hold. 

Let $(X,\phi)$ be a non-wandering Smale space and $(E,e)$ its projective cover. Note that $\phi$ induces a self-homeomorphism $\tilde{\phi}$ of $E$ and $e$ intertwines $\tilde{\phi},\phi$. Consider the totally disconnected space $\Sigma_{0,0}(\pi)$ associated to a choice of $s/u$-bijective pair $\pi$ (this is the correct analogue of $\Sigma_0$ when $X$ is not totally disconnected along the stable direction). The difference between $E$ and $\Sigma_{0,0}(\pi)$ can be recast in terms of the dependence of the latter space on $\pi$. This suggests that in order to make sense of projectivity in the context of Smale spaces we ought to consider \emph{all} $s/u$-bijective pairs at the same time. 

The discussion on projectivity will inevitably bring us outside the category of Smale spaces (e.g., extremally disconnected spaces are not metrizable, unless they are discrete), therefore the following setup is in the context of (invertible) dynamical systems. See also Remark \ref{rem:mk} below.
An open set in a space $X$ is called regular if it is the interior of its closure. A \emph{regular partition} $\P$ of $X$ is a finite collection of disjoint regular opens in $X$ whose union is dense.

Let $(X,\phi)$ be an invertible dynamical system and $\P$ a regular partition of $X$. View $\P$ as an alphabet and let $a_1a_2\cdots a_n$ be a word. We say this word is \emph{allowed} if $\cap_{i=1}^n \phi^{-i}(a_i)\neq \emptyset$ and let $L_{\P}$ be the family of allowed words. It can be checked \cite[Section 6.5]{lind:marcus} that $L_\P$ is the language of a shift space that we denote $\Sigma_\P$. Note that for each $x\in\Sigma_\P$ and $n\in \N$, the set
\[
D_n(x)=\bigcap_{i=-n}^n\phi^{-i}(x_i)\subseteq X
\]
is nonempty.

\begin{defi}
We say that $\P$ is a \emph{symbolic presentation} of $(X,\phi)$ if for every $x\in \Sigma_\P$ the set $\cap_{n=0}^\infty \overline{D_n(x)}$ consists of exactly one point. We call $\P$ a \emph{Markov partition} if $\Sigma_\P$ is a subshift of finite type.
\end{defi}

Other definitions of Markov partitions are common in the literature, e.g.,~\cite{bow:mA}.
Notice that the set of regular partitions is directed: we write $\P_1\leq \P_2$ if $\P_2$ is a refinement of $\P_1$, i.e., each member of $\P_2$ is contained in a member of $\P_1$. Given partitions $\P_1,\P_2$ we can define an upper bound $\P_1\cap \P_2$, obtained by taking pairwise intersections of elements from each partition.

If $(X,\phi)$ admits a symbolic presentation $\P_1$, then given any regular partition $\P_2$ we have that $\P_1\cap \P_2$ is again a symbolic presentation. In other words, once a symbolic presentation exists, we can guarantee that the family of symbolic presentations is cofinal among all regular partitions.

Associated to $\P_1$ we get a factor map (i.e., an equivariant surjection) $\pi_{\P_1}\colon\Sigma_{\P_1}\to X$ (see \cite[Proposition 6.5.8]{lind:marcus}). If $\P_2$ is a refinement of $\P_1$, then $\pi_{\P_2}\colon\Sigma_{\P_2}\to X$ is a factor map which factors through $\Sigma_{\P_1}$. Indeed if we view $\P_1$ and $\P_2$ as alphabets, there is a code $\mu_{\P_1,\P_2}$ which assigns to each letter $a\in \P_2$ the unique letter $b\in \P_1$ such that $a\subseteq b$, and $\pi_{\P_2}=\pi_{\P_1}\circ\mu_{\P_1,\P_2}$. 

As a result, if $I$ denotes the family of symbolic presentations of $(X,\phi)$ (assuming it is nonempty), then $(\Sigma_i,\mu_{ij},\pi_i)_{i\leq j\in I}$ defines a projective system in the category of dynamical systems over $X$. Let $E$ be the inverse limit of $(\Sigma_i,\mu_{ij},\pi_i)_{i\leq j\in I}$. Since $E\subseteq \prod_i\Sigma_i$, the shift map $\sigma$ applied componentwise turns $E$ into a dynamical system. Given $\P\in I$, denote by $p_\P$ the canonical projection $E\to \Sigma_\P$.

The case of Smale spaces is as follows. A non-wandering Smale space $(X,\phi)$ always admits a Markov partition \cite[Section 7]{ruelle:thermo}. If we denote such partition by $\M$, then $\Sigma_\M$ is a subshift of finite type endowed with an \emph{almost one-to-one} factor map $\pi_\M\colon \Sigma_\M\to X$ (i.e., an equivariant surjection that is finite-to-one, and the set of points in $X$ with single preimage is a dense $G_\delta$). If $\P$ is a refinement of $\M$, then $\pi_\P\colon\Sigma_\P\to X$ is an almost one-to-one factor map which factors through $\Sigma_\M$. As a result, in this case we can take $I$ to be the family of refinements of $\M$.

\begin{rema}\label{rem:mk}
It is worth noting that $\{\Sigma_i\}_{i\in I}$ is a collection of shift spaces that are not necessarily of finite type (in particular, they are not Smale spaces). That is because the refinement of a Markov partition is not a Markov partition (in general). It is unclear to the author if there are conditions under which a Smale space admits a cofinal collection of Markov partitions.
\end{rema}

\begin{theo}\label{thm:proj}
Let $(X,\phi)$ be a dynamical system which admits a symbolic presentation $\P$. Suppose $(\Sigma_i,\mu_{ij},\pi_i)_{i\leq j\in I}$ is the projective system associated to the collection of symbolic presentations of $(X,\phi)$ and denote by $(E,\sigma)$ the associated inverse limit. Then $(E,\sigma)$ is a projective cover of $(X,\phi)$ and the map $e\colon E \to X$ is given by the composition
\[
\xymatrix{E \ar[r]^-{p_\P} & \Sigma_\P \ar[r]^-{\pi_\P} & X}.
\]
\end{theo}
\begin{proof}
Let $J$ be the family of regular partitions of $X$. Given $\P\in J$, denote by $X(\P)$ the topological space given by the disjoint union $\cup_{Y\in \P}\overline{Y}$. Then by \cite[Proposition 17]{rump:proj} we have that
\[
E^\prime=\varprojlim_{j\in J} (X(j),f_{jk})
\]
is a projective cover of $X$ (here $f_{jk}\colon X(k)\twoheadrightarrow X(j)$ when $j\leq k$ is the obvious surjection induced by the refinement). First of all we notice that $I$ is cofinal in $J$ so that the limit can be taken over the index set $I$. Secondly, notice that for each $i\in I$ there is a natural surjection $p_i\colon X(i) \to X$. We claim that $\pi_i\colon \Sigma_i\to X$ factors through $p_i$. Indeed, note that if $x\in \Sigma_i$, then $\pi_i(x)$ belongs to $\overline{x_0}\in i$ and $\pi_i(x)$ admits a unique lift $\tilde{x}\in \overline{x_0}\subseteq X(i)$. Define $\tilde{\pi}_i(x)=\tilde{x}$ and by construction $\pi_i=p_i\circ\tilde{\pi}_i$.

It is easy to check that ($i\leq j$)
\[
\xymatrix{ \Sigma_j\ar[r]^-{\mu_{ij}} \ar[d]^-{\tilde{\pi}_j} & \Sigma_i \ar[d]^-{\tilde{\pi}_i}\\
X(j) \ar[r]^-{f_{ij}} & X(i)}
\]
is a commuting diagram so that $\{\tilde{\pi}_i\}_{i\in I}$ induces a (continuous) map of spaces $\tilde{\pi} \colon E \to~E^\prime$. Since $\tilde{\pi}$ is a map of compact Hausdorff spaces, we only need to show it is bijective in order to get the required homeomorphism $E\cong E^\prime$. In fact, it is sufficient to show that it is one-to-one, because $\tilde{\pi}(E)\subseteq E^\prime$ is a closed set mapping onto $X$, thus by irreducibility $\tilde{\pi}(E)=E^\prime$.

Suppose $x,y\in E, x\neq y$, so there is $i\in I$ with $x_i\neq y_i$. Recall that $x_i$ and $y_i$ are bi-infinite sequences in $\Sigma_i$, let us denote their components by $(x_i^k)_{k\in\Z},(y_i^k)_{k\in\Z}$. 

There is $m\in \Z$ with $x_i^m\neq y_i^m$. Note that $\phi^{-m}(i)$ is also a symbolic presentation, and if we set $\alpha=i\cap \phi^{-m}(i)$ we have $i\leq \alpha$, thus there are elements $x_\alpha,y_\alpha\in \Sigma_\alpha$, appearing at the $\alpha$-th component of respectively $x,y$, and satisfying $\mu_{i\alpha}(x_\alpha)=x_i, \mu_{i\alpha}(y_\alpha)=y_i$.

We claim $\tilde{\pi}_\alpha(x_\alpha)\neq \tilde{\pi}_\alpha(y_\alpha)$. In fact, there are $A_x,B_x,A_y,B_y\in i$ with
\begin{align*}
x_\alpha^0 &=x_i^0 \cap \phi^{-m}(A_x)  & x_\alpha^m&=x_i^m\cap\phi^{-m}(B_x)\\
y_\alpha^0 &=x_i^0 \cap \phi^{-m}(A_y)  & x_\alpha^m&=x_i^m\cap\phi^{-m}(B_y)\\
\end{align*}
\vspace{-8ex}\begin{align*}
x_i^0 \cap \phi^{-m}(A_x)\cap \phi^{-m}(x_i^m)\cap\phi^{-2m}(B_x)&\neq\emptyset \\
y_i^0 \cap \phi^{-m}(A_y)\cap \phi^{-m}(y_i^m)\cap\phi^{-2m}(B_y)&\neq\emptyset.
\end{align*}
From the above we derive $A_x=x_i^m,A_y=y_i^m$ and in particular $A_x\neq A_y$. But by definition 
\begin{align*}
\tilde{\pi}_\alpha(x_\alpha) &\in \overline{x_i^0 \cap \phi^{-m}(A_x)}\subseteq X(\alpha)\\
\tilde{\pi}_\alpha(y_\alpha) &\in \overline{y_i^0 \cap \phi^{-m}(A_y)}\subseteq X(\alpha)
\end{align*}
so $\tilde{\pi}_\alpha(x_\alpha)$ cannot be equal to $\tilde{\pi}_\alpha(y_\alpha)$. This proves injectivity of $\tilde{\pi}$ and concludes the proof.
\end{proof}

\begin{rema}
At first sight, it it reasonable to view $E$ as the ``universal'' version of the spaces of the form $\Sigma_{0,0}(\pi)$. In the same spirit, one could think of defining a ``universal $s/u$-bijective pair'' $\pi=(E_s,\tilde{\psi},e_s,E_u,\tilde{\zeta},e_u)$, where $E_s$ and $E_u$ would be projective with respect to $s$-bijective and $u$-bijective maps.

The first step towards this program would be applying Putnam's lifting theorem \cite{put:lift} to the projective system $\{\Sigma_i\}_{i\in I}$ of Theorem \ref{thm:proj} (assuming the system, or a cofinal replacement, consists entirely of shifts of finite type). Unfortunately, in order to lift the entire (infinite) system, limits of spaces are necessary, thus we run once again into the problem that these limits are not Smale spaces, and the notions of $s$- and $u$-bijective maps don't work well in this context. This suggests that, if one desires importing the machinery of homological algebra in the setting of Smale spaces, the ambient category should be chosen with care. A good candidate for this category might be the equivariant (with respect to the stable or unstable equivalence relation) $\mathrm{KK}$-category, but this idea will not be pursued in the present paper.
\end{rema}

\backmatter

\bibliographystyle{smfplain}
\bibliography{BibliographyBST}   

\providecommand{\bysame}{\leavevmode ---\ }
\providecommand{\og}{``}
\providecommand{\fg}{''}
\providecommand{\smfandname}{\&}
\providecommand{\smfedsname}{\'eds.}
\providecommand{\smfedname}{\'ed.}
\providecommand{\smfmastersthesisname}{M\'emoire}
\providecommand{\smfphdthesisname}{Th\`ese}
\begin{thebibliography}{10}

\bibitem{renroch:amgrp}
{\scshape C.~Anantharaman-Delaroche {\normalfont \smfandname} J.~Renault} --
  \emph{Amenable groupoids}, Monographies de L'Enseignement Math\'ematique
  [Monographs of L'Enseignement Math\'ematique], vol.~36, L'Enseignement
  Math\'ematique, Geneva, 2000, With a foreword by Georges Skandalis and
  Appendix B by E. Germain.

\bibitem{bow:mA}
{\scshape R.~Bowen} -- {\og Markov partitions for {A}xiom {${\rm A}$}
  diffeomorphisms\fg}, \emph{Amer. J. Math.} \textbf{92} (1970), p.~725--747.

\bibitem{brown:ozawa}
{\scshape N.~P. Brown {\normalfont \smfandname} N.~Ozawa} --
  \emph{{$C^*$}-algebras and finite-dimensional approximations}, Graduate
  Studies in Mathematics, vol.~88, American Mathematical Society, Providence,
  RI, 2008.

\bibitem{deeley:strung}
{\scshape R.~J. Deeley {\normalfont \smfandname} K.~R. Strung} -- {\og Nuclear
  dimension and classification of {$\rm C^*$}-algebras associated to {S}male
  spaces\fg}, \emph{Trans. Amer. Math. Soc.} \textbf{370} (2018), no.~5,
  p.~3467--3485.

\bibitem{dyck:proj}
{\scshape R.~Dyckhoff} -- {\og Projective resolutions of topological
  spaces\fg}, \emph{J. Pure Appl. Algebra} \textbf{7} (1976), no.~1,
  p.~115--119.

\bibitem{glea:proj}
{\scshape A.~M. Gleason} -- {\og Projective topological spaces\fg},
  \emph{Illinois J. Math.} \textbf{2} (1958), p.~482--489.

\bibitem{grandis:symsets}
{\scshape M.~Grandis} -- {\og Finite sets and symmetric simplicial sets\fg},
  \emph{Theory Appl. Categ.} \textbf{8} (2001), p.~244--252.

\bibitem{johnson:groups}
{\scshape D.~L. Johnson} -- \emph{Topics in the theory of group presentations},
  London Mathematical Society Lecture Note Series, Cambridge University Press,
  1980.

\bibitem{krieger:inv}
{\scshape W.~Krieger} -- {\og On dimension functions and topological {M}arkov
  chains\fg}, \emph{Invent. Math.} \textbf{56} (1980), no.~3, p.~239--250.

\bibitem{lind:marcus}
{\scshape D.~Lind {\normalfont \smfandname} B.~Marcus} -- \emph{An introduction
  to symbolic dynamics and coding}, Cambridge University Press, Cambridge,
  1995.

\bibitem{murewi:morita}
{\scshape P.~S. Muhly, J.~N. Renault {\normalfont \smfandname} D.~P. Williams}
  -- {\og Equivalence and isomorphism for groupoid {$C^\ast$}-algebras\fg},
  \emph{J. Operator Theory} \textbf{17} (1987), no.~1, p.~3--22.

\bibitem{put:algSmale}
{\scshape I.~F. Putnam} -- {\og {$C^*$}-algebras from {S}male spaces\fg},
  \emph{Canad. J. Math.} \textbf{48} (1996), no.~1, p.~175--195.

\bibitem{put:funct}
{\scshape I.~F. Putnam} -- {\og Functoriality of the {$C^*$}‐algebras
  associated with hyperbolic dynamical systems\fg}, \emph{Journal of the London
  Mathematical Society} \textbf{62} (2000), no.~3, p.~873--884.

\bibitem{put:lift}
{\scshape I.~F. Putnam} -- {\og Lifting factor maps to resolving maps\fg},
  \emph{Israel J. Math.} \textbf{146} (2005), p.~253--280.

\bibitem{put:HoSmale}
\bysame , {\og A homology theory for {S}male spaces\fg}, \emph{Mem. Amer. Math.
  Soc.} \textbf{232} (2014), no.~1094, p.~viii+122.

\bibitem{put:spiel}
{\scshape I.~F. Putnam {\normalfont \smfandname} J.~Spielberg} -- {\og The
  structure of {$C^*$}-algebras associated with hyperbolic dynamical
  systems\fg}, \emph{Journal of Functional Analysis} \textbf{163} (1999),
  no.~2, p.~279 -- 299.

\bibitem{ruelle:thermo}
{\scshape D.~Ruelle} -- \emph{Thermodynamic formalism}, second \smfedname,
  Cambridge Mathematical Library, Cambridge University Press, Cambridge, 2004,
  The mathematical structures of equilibrium statistical mechanics.

\bibitem{rump:proj}
{\scshape W.~Rump} -- {\og The absolute of a topological space and its
  application to abelian {$l$}-groups\fg}, \emph{Appl. Categ. Structures}
  \textbf{17} (2009), no.~2, p.~153--174.

\bibitem{smale:A}
{\scshape S.~Smale} -- {\og Differentiable dynamical systems\fg}, \emph{Bull.
  Amer. Math. Soc.} \textbf{73} (1967), p.~747--817.

\bibitem{thomsen:smale}
{\scshape K.~Thomsen} -- {\og {$C^*$}-algebras of homoclinic and heteroclinic
  structure in expansive dynamics\fg}, \emph{Mem. Amer. Math. Soc.}
  \textbf{206} (2010), p.~vii+122.

\bibitem{wei:homalg}
{\scshape C.~A. Weibel} -- \emph{An introduction to homological algebra},
  Cambridge Studies in Advanced Mathematics, vol.~38, Cambridge University
  Press, Cambridge, 1994.

\end{thebibliography}
\end{document}